\documentclass[12pt,leqno]{amsart}
\usepackage{subfigure}
\usepackage{epsfig}
\usepackage{amsmath, amsthm, amssymb}
\usepackage[colorlinks, citecolor=blue, linkcolor=red]{hyperref}

\newcommand{\be}{\begin{eqnarray}}
\newcommand{\ee}{\end{eqnarray}}

\newcommand{\RR}{\mathbb{R}}
\newcommand{\HH}{\mathbb{H}}
\theoremstyle{plain}

\newtheorem{theorem}{\textbf{Theorem}}[section]
\newtheorem{lemma}[theorem]{\textbf{Lemma}}
\newtheorem{proposition}[theorem]{\textbf{Proposition}}
\newtheorem{cor}[theorem]{\textbf{Corollary}}

\newtheorem{rem}[theorem]{\textbf{Remark}}
\theoremstyle{remark}

\newtheorem{ackn}{Acknowledgments\!\!}

\newcommand{\ricc}{\mathrm{Ric}}
\newcommand{\vol}{\mathrm{Vol}}

\title[The Poisson equation on manifolds]
{The Poisson equation \\on Riemannian manifolds with\\ weighted Poincar\'e inequality at infinity}

\author[G. Catino]{Giovanni Catino}
\address[Giovanni Catino]{Dipartimento di Matematica, Politecnico di Milano, Piazza Leonardo da Vinci 32, 20133 Milano, Italy}
\email[]{giovanni.catino@polimi.it}

\author[D. D. Monticelli]{Dario D. Monticelli}
\address[Dario Monticelli]{Dipartimento di Matematica, Politecnico di Milano, Piazza Leonardo da Vinci 32, 20133 Milano, Italy}
\email[]{dario.monticelli@polimi.it}

\author[F. Punzo]{Fabio Punzo}
\address[Fabio Punzo]{Dipartimento di Matematica, Politecnico di Milano, Piazza Leonardo da Vinci 32, 20133 Milano, Italy}
\email[]{fabio.punzo@polimi.it}



\begin{document}

\begin{abstract}
We prove an existence result for the Poisson equation on non-compact Riemannian manifolds satisfying weighted Poincar\'e inequalities outside compact sets. Our result applies to a large class of manifolds including, for instance, all non-parabolic manifolds with minimal positive Green's function vanishing at infinity. On the source function we assume a sharp pointwise decay depending on the weight appearing in the Poincar\'e inequality and on the behavior of the Ricci curvature at infinity. We do not require any curvature or spectral assumptions on the manifold.
\end{abstract}

\maketitle

\section{Introduction}

The existence of solutions to the Poisson equation
$$
\Delta u = f
$$
on a complete Riemannian manifold $(M, g)$, for a given function $f$ on $M$, is a classical problem which has been the object of deep interest in the literature. Malgrange \cite{mal} obtained solvability of the Poisson equation for any smooth function $f$ with compact support, as a consequence of the existence of a Green's function for $-\Delta$ on every complete Riemannian manifold. Under integrability assumptions on $f$, existence of solutions have been established by Strichartz \cite{str} and Ni-Shi-Tam \cite[Theorem 3.2]{nst} (see also \cite[Lemma 2.3]{ni}). Moreover, in the same paper, the authors proved an existence result for the Poisson problem on manifolds with non-negative Ricci curvature under a sharp integral assumption involving suitable averages of $f$. This condition in particular is satisfied if
$$
|f(x)|\leq \frac{C}{\big(1+r(x)\big)^{\alpha}}
$$
for some $C>0$ and $\alpha>2$, where $r(x):=\operatorname{dist}(x,p)$ is the distance function of any $x\in M$ from a fixed reference point $p\in M$. In fact, they proved a more general result where the decay rate of $f$ is just assumed to be of order $1+\varepsilon$. Note that this result is sharp on the flat space $\RR^{n}$.

From now on let us consider solutions $u$ of the Poisson equation $\Delta u=f$ which can be represented as
$$
u(x)=\int_{M} G(x,y)f(y)\,dy\,,
$$
where $G(x,y)$ is a Green's function of $-\Delta$ on $M$ (see Section \ref{sect-prel} for further details). Muntenau-Sesum \cite{ms} addressed the case of manifolds with positive spectrum, i.e. $\lambda_1(M)>0$, and Ricci curvature bounded from below, obtaining existence of solutions under the pointwise decay assumption
$$
|f(x)|\leq \frac{C}{\big(1+r(x)\big)^{\alpha}}
$$
for some $C>0$ and $\alpha>1$. Note that this result is sharp on $\HH^{n}$. Their proof relies on very precise integral estimates on the minimal positive Green's function, which are inspired by the work of Li-Wang \cite{liwa1}.

In \cite{cmp} the authors generalized the result in \cite{ms}, obtaining existence of solutions on manifolds with positive essential spectrum, i.e. $\lambda_1^{\text{ess}}(M)>0$,  for source functions $f$ satisfying
$$
\sum_{m=1}^{\infty}\frac{\theta_{R}(m+1)-\theta_{R}(m)}{\lambda_{1}\left(M\setminus B_{m}(p)\right)}\sup_{M\setminus B_m(p)}|f| < \infty,
$$
for any $R>0$, where $\theta_{R}(m)$ is a function related to a lower bound on the Ricci curvature, locally on geodesic balls with center $p$ and radius $2R+m$. In particular, the authors showed in \cite[Corollary 1.3]{cmp} existence of solutions on Cartan-Hadamard manifolds with strictly negative Ricci curvature, whenever
$$
-C\big(1+r(x)\big)^{\gamma_{1}} \leq \ricc \leq -\frac{1}{C}\big(1+r(x)\big)^{\gamma_{2}} ,\quad |f (x)| \leq \frac{C}{\big(1+r(x)\big)^{\alpha}},
$$
for some $C>0$ and $\gamma_{1},\gamma_{2}\geq 0$ with $\alpha>1+\frac{\gamma_{1}}{2}-\gamma_{2}$.

Observe that the results in \cite{ms} and \cite{cmp} cannot be used whenever the Ricci curvature tends to zero at infinity fast enough (see \cite{jpw}) since, in this case, one has $\lambda_1^{\text{ess}}(M)=0$ (and so $\lambda_1(M)=0$). In particular the case of $\RR^n$ is not covered. On the other hand, the result in \cite{nst} does not apply on manifolds with negative curvature. The purpose of our paper is to obtain a general result which includes, as special cases, both manifolds with strictly negative curvature and manifolds with Ricci curvature vanishing at infinity. Moreover, our result is sharp on spherically symmetric manifolds, and in particular on $\RR^n$ and $\HH^n$.

Note that the condition $\lambda_1(M)>0$ is equivalent to the validity of the Poincar\'e inequality
$$
\lambda_1(M)\int_M u^2\, dV \leq \int_M |\nabla u|^2\,dV
$$
for any $u\in C^\infty_c(M)$. On the other hand, one has positive essential spectrum if and only if, for some compact subset $K\subset M$, one has $\lambda_1(M \setminus K)>0$ and
$$
\lambda_1(M \setminus K)\int_M u^2\, dV \leq \int_M |\nabla u|^2\,dV
$$
for any $u\in C^\infty_c(M\setminus K)$. Generalizing the previous inequalities, one says that $(M,g)$ satisfies a {\em weighted Poincar\'e inequality}  with a non-negative weight function $\rho$ if
\begin{equation}\label{wpi2}
\int_M \rho \,v^2\, dV \leq \int_M |\nabla v|^2 \,dV
\end{equation}
for every $v\in C^\infty_c(M)$. If for any $R\geq R_0>0$ there exists a non-negative function $\rho_R$ such that \eqref{wpi2} holds for every $v\in C^\infty_c(M\setminus B_R(p))$ and for $\rho\equiv\rho_R$, we say that $(M,g)$ satisfies a {\em weighted Poincar\'e inequality at infinity}. In addition, inspired by \cite{liwa1}, we say that $(M,g)$ satisfies the property $\left(\mathcal{P}^{\infty}_{\rho_R}\right)$ if a weighted Poincar\'e inequality at infinity holds for the family of weights $\rho_R$ and the conformal $\rho_R$-metric defined by
$$
g_{\rho_R} := \rho_R\, g
$$
is complete for every $R\geq R_0$. The validity of a weighted Poincar\'e inequality on some classes of manifolds has been investigated in the literature. It is well known that on $\RR^n$ inequality \eqref{wpi2} holds with $\rho(x)=\frac{(n-2)^2}{4}\frac{1}{r^2(x)}$. It is also called {\em Hardy inequality}. More in general, it holds on every Cartan-Hadamard manifold with $\rho(x)=\frac{C}{r^2(x)}$, for some $C>0$ (see \cite{car} and \cite{gri} for some refinement of this result).

In order to state our main results, we need to introduce a (increasing) function $\omega(s)$ related to the value of the Ricci curvature on the annulus $B_{\frac{5}{4}s}(p)\setminus B_{\frac{3}{4}s}(p)$ (see \eqref{eq127} for the precise definition). In this paper we prove the following result.

\begin{theorem} \label{teo1} Let $(M,g)$ be a complete non-compact Riemannian manifold satisfying the property $\left(\mathcal{P}^\infty_{\rho_R}\right)$ and let $f$ be a locally H\"older continuous function on $M$. If
$$
\sum_{m}^{\infty}\Big(\omega(m+1)-\omega(m)+1\Big)\sup_{M\setminus B_m(p)}\left|\frac{f}{\rho_m}\right| < \infty,
$$
then the Poisson equation
\begin{equation*}
\Delta u=f \quad\hbox{in   } M
\end{equation*}
admits a classical solution $u$.
\end{theorem}
Assume that $\lambda_1^{\text{ess}}(M)>0$ and
$$
\ricc \geq -C\big(1+r(x)\big)^{\gamma}
$$
for some $\gamma\geq 0$. Then it is direct to see that
$$
\omega(m+1)-\omega(m)\sim C\Big(\theta_{R}(m+1)-\theta_{R}(m)\Big) \sim C m^{\frac{\gamma}{2}}
$$
for every $R>0$ and the property $\left(\mathcal{P}^\infty_{\rho_R}\right)$ holds for every $R$ with $\rho_R(x)=\lambda_1(M\setminus B_R(p))$. Thus
$$
\Big(\omega(m+1)-\omega(m)+1\Big)\sup_{M\setminus B_m(p)}\left|\frac{f}{\rho_m}\right| \sim  C\, \frac{\theta_{R}(m+1)-\theta_{R}(m)}{\lambda_{1}\left(M\setminus B_{m}(p)\right)}\sup_{M\setminus B_m(p)}|f| \,,
$$
therefore our result is in accordance with those in \cite{ms} and \cite{cmp}.

We recall that by \cite[Corollary 1.4, Lemma 1.5]{liwa1} the validity of a weighted Poincar\'e inequality \eqref{wpi2} on $M$ implies the non-parabolicity of the manifold; on the contrary, if $(M,g)$ is non-parabolic, then a weighted Poincar\'e inequality holds on $M$, with weight $$\rho(x):=\frac{|\nabla G(p,x)|^2}{4 G^2(p,x)},$$ where $G$ is the minimal positive Green's function on $(M,g)$. Exploiting this result, using similar techniques as in Theorem \ref{teo1}, we obtain the following refined result on complete non-compact non-parabolic manifolds.

\begin{theorem} \label{teo2} Let $(M,g)$ be a complete non-compact non-parabolic Riemannian manifold with minimal positive Green's function $G$. Let $\rho(x)=\frac{|\nabla G(p,x)|^2}{4 G^2(p,x)}$ and let $f$ be a locally H\"older continuous function on $M$. If
$$
\sum_{m}^{\infty}\Big(\omega(m+1)-\omega(m)\Big)\sup_{M\setminus B_m(p)}\left|\frac{f}{\rho}\right| < \infty,
$$
then the Poisson equation
\begin{equation*}
\Delta u=f \quad\hbox{in   } M
\end{equation*}
admits a classical solution $u$.
\end{theorem}

\begin{rem}
We explicitly observe that in Theorem \ref{teo2} the completeness of the conformal metric $g_\rho=\rho g$ is not required. As it was observed in \cite{liwa1}, the completeness of $g_\rho$ would hold if $G(p,x)\to 0$ as $r(x)\to \infty$, a condition that we do not need to assume here.
\end{rem}

It is well-known that $\RR^n$ is a non-parabolic manifold if $n\geq3$, with minimal positive Green's function $G(x,y)=\frac{c_n}{|x-y|^{n-2}}$ for some positive constant $c_n$. Moreover the weighted Poincar\'{e} -- Hardy's inequality holds on $\RR^n$ with  $$\rho(x)=\frac{|\nabla G(0,x)|^2}{4 G^2(0,x)}=\frac{(n-2)^2}{4}\frac{1}{|x|^2}.$$ In this case, using the definition \eqref{eq127} of the function $\omega(s)$, it is easy to see that
$$
\omega(m+1)-\omega(m)\sim C \log \left(1+\frac{1}{m}\right) \sim \frac{C}{m}\,.
$$
Hence we can apply Theorem \ref{teo2}, with
$$
\Big(\omega(m+1)-\omega(m)\Big)\sup_{M\setminus B_m(p)}\left|\frac{f}{\rho_m}\right| \sim C\, m\, \sup_{M\setminus B_m(p)}\left|f\right|
$$
and the convergence of the series follows, whenever $|f(x)|\leq C/(1+r(x))^\alpha$ for some $\alpha>2$. This condition is optimal, as it can be easily verified by explicit computations.  

In general, concerning Cartan-Hadamard manifolds, by using Theorem \ref{teo1} we improve \cite[Corollary 1.3]{cmp} allowing the Ricci curvature to approach zero at infinity.

\begin{cor}\label{cor-2} Let $(M,g)$ be a Cartan-Hadamard manifold and let $f$ be a locally H\"older continuous, bounded function on $M$. If
$$
-C\big(1+r(x)\big)^{\gamma_{1}} \leq \ricc \leq -\frac{1}{C}\big(1+r(x)\big)^{\gamma_{2}} ,\quad |f (x)| \leq \frac{C}{\big(1+r(x)\big)^{\alpha}},
$$
for some $C\geq 1$, $\gamma_1,\gamma_2\in\RR$, $\gamma_{1}\geq\gamma_{2}$, $\gamma_1\geq0$ and $\alpha$ satisfying
$$
\alpha >
\begin{cases}
1+\frac{\gamma_1}{2}-\gamma_2 &\quad\hbox{if } \gamma_2\geq-2 \\
3+\frac{\gamma_1}{2}  &\quad\hbox{if } \gamma_2< -2 
\end{cases}
$$
then the Poisson equation
\begin{equation*}
\Delta u=f \quad\hbox{in   } M
\end{equation*}
admits a classical solution $u$.
\end{cor}
\begin{rem}\label{rem-rot}
In the special case $\gamma_{1}=\gamma_{2}=\gamma\geq 0$ the condition on $\alpha$ in the previous corollary becomes
$$
\alpha >
\begin{cases}
1-\frac{\gamma}{2}&\quad\hbox{if } \gamma\geq-2 \\
2 &\quad\hbox{if } \gamma< -2\,.
\end{cases}
$$
In particular in $(M,g)$ is the standard hyperbolic space $\HH^n$, then $\gamma=0$. Thus we need that $\alpha>1$ and this condition is sharp as observed above. We will consider also the case $\gamma<0$ in the Subsection \ref{ssu} on model manifolds.
\end{rem}

The paper is organized as follows: in Section \ref{sect-prel} we collect some preliminary results and we define precisely the function $\omega$; in Section \ref{sec-grad} we prove a refined local gradient estimates for positive harmonic functions; in Section \ref{sec-est} we prove key estimates on the positive minimal Green's function $G(x,y)$ of a non-parabolic manifold, by means of the property $\left(\mathcal{P}^{\infty}_{\rho_R}\right)$; in Section \ref{sec-proofs} we prove Theorem \ref{teo1}; finally in Section \ref{sec-ex} we prove Corollary \ref{cor-2} and show the optimality of the assumption in Theorem \ref{teo2} for rotationally symmetric manifolds.

\

Finally we note that some results concerning the Poisson equation on some manifolds satisfying a weighted Poincar\'e inequality have been very recently obtained in \cite{msw2}. However their assumptions and results apparently are completely different to ours.

\

\section{Preliminaries}

\label{sect-prel}

Let $(M,g)$ be a complete non-compact $n$-dimensional Riemannian manifold. For any $x\in M$ and $R>0$, we denote by $B_{R}(x)$ the geodesic ball of radius R with centre $x$ and let $\vol(B_{R}(x))$ be its volume. We denote by $\ricc$ the Ricci curvature of $g$. For any $x \in M$, let $\mu(x)$ be the smallest eigenvalue of $\ricc$ at $x$. Thus, for any $V\in T_{x}M$ with $|V|=1$, $\ricc(V,V)(x) \geq \mu(x)$ and we have $\mu(x)\geq -\omega(r(x))$ for some $\omega\in C([0,\infty))$, $\omega\geq 0$. Hence, for any $x\in M$, we have
\begin{equation}\label{eq3}
\ricc(V,V)(x) \geq -(n-1) \frac{\varphi''(r(x))}{\varphi(r(x))},
\end{equation}
for some $\varphi\in C^{\infty}((0,\infty))\cap C^{1}([0,\infty))$ with $\varphi(0)=0$ and $\varphi'(0)=1$. Note that $\varphi,\varphi',\varphi''$ are positive in $(0,\infty)$. We set
$$
K_R(x):=\sup_{y\in B_{r(x)+R}\setminus B_{r(x)-R}}\frac{\varphi''(r(y))}{\varphi(r(y))}
$$
for $r(x)>R>1$;
$$
I_R(x):=\begin{cases}
\sqrt{K_R(x)}\coth\left(\sqrt{K_R(x)} R/2\right)&\text{if }\,K_R(x)>0 \\
\frac{2}{R} &\text{if }\,K_R(x)=0;
\end{cases}
$$
\begin{align}\label{defQ}
Q_{R}(x):=\max\left\{K_R(x), \frac{I_R(x)}{R}, \frac{1}{R^2}\right\}.
\end{align}
Note that $Q_{R}(x)\equiv Q_{R}(r(x))$. For any $z\in M$, let $\gamma$ be the minimal geodesic connecting $p$ to $z$. We define the function
\begin{equation}\label{eq127}
\omega(z)=\omega(r(z)):=\int_a^{r(z)} \sqrt{Q_{\frac{r((\gamma(s))}{4}}(r(\gamma(s))}\,ds,
\end{equation}
for a given $a>0$. Note that $t\mapsto\omega(t)$ is increasing and so invertible.

Under $\eqref{eq3}$, we know that
\begin{equation}\label{eq6}
\vol(B_{R}(p)) \leq C \int_{0}^{R}\varphi^{n-1}(\xi)\,d\xi.
\end{equation}
Moreover, let $\operatorname{Cut}(p)$ be the {\em cut locus} of $p\in M$.

It is known that every complete Riemannian manifold admits a Green's function (see \cite{mal}), i.e. a smooth function defined in $(M\times M)\setminus \{(x,y)\in M\times M:\,x=y\} $ such that $G(x,y)=G(y,x)$ and $\Delta_{y} G(x,y)=-\delta_{x}(y)$. We say that $(M,g)$ is non-parabolic if there exists a minimal positive Green's function $G(x,y)$ on $(M,g)$, and parabolic otherwise.

We say that $(M,g)$ satisfies a {\em weighted Poincar\'e inequality}  with a non-negative weight function $\rho$ if
\begin{equation}\label{wpi}
\int_M \rho \,v^2\, dV \leq \int_M |\nabla v|^2 \,dV
\end{equation}
for every $v\in C^\infty_c(M)$.  If for any $R\geq R_0>0$ there exists a non-negative function $\rho_R$ such that \eqref{wpi2} holds for every $v\in C^\infty_c(M\setminus B_R(p))$ and for $\rho\equiv\rho_R$, we say that $(M,g)$ satisfies a {\em weighted Poincar\'e inequality at infinity}. In addition, inspired by \cite{liwa1}, we say that $(M,g)$ satisfies the property $\left(\mathcal{P}^{\infty}_{\rho_R}\right)$ if a weighted Poincar\'e inequality at infinity holds for the family of weights $\rho_R$ and the conformal $\rho_R$-metric defined by
$$
g_\rho := \rho_R\, g
$$
is complete.  With this metric we consider the $\rho$-distance function
$$
r_\rho (x,y)=\inf_{\gamma} \, l_\rho (\gamma)
$$
where the infimum of the lengths  is taken over all curves joining $x$ and $y$, with respect to the metric $g_\rho$. For a fixed point $p\in M$, we denote by $$r_\rho(x) = r_\rho (p,x).$$ Note that $|\nabla r_\rho (x)|^2 = \rho(x)$. Finally, we denote by $$B^\rho_R(p)=\{x \in M: r_\rho(x)\leq R\}.$$ 

Let $\lambda_{1}(M)$ be the bottom of the $L^{2}$-spectrum of $-\Delta$. It is known that $\lambda_{1}(M)\in[0,+\infty)$ and it is given by the variational formula
$$
\lambda_{1}(M) = \inf_{v\in C^{\infty}_{c}(M)}\frac{\int_{M}|\nabla v|^{2}\,dV}{\int_{M}v^{2}\,dV}\,.
$$
If $\lambda_{1}(M)>0$, then $(M,g)$ is non-parabolic (see \cite[Proposition 10.1]{gri1}). Whenever $(M,g)$ is non-parabolic, let $G_{R}(x,y)$ be the Green's function of $-\Delta$ in $B_{R}(z)$ satisfying zero Dirichlet boundary conditions on $\partial B_{R}(z)$, for some $z\in M$. We have that $R\mapsto G_{R}(x,y)$ is increasing and, for any $x,y\in M$,
\begin{equation}\label{eq9}
G(x,y) = \lim_{R\to\infty} G_{R}(x,y),
\end{equation}
locally uniformly in $(M\times M)\setminus \{(x,y)\in M\times M:\,x=y\} $.
We define $\lambda_{1}(\Omega)$, with $\Omega$ an open subset of $M$, to be the first eigenvalue of $-\Delta$ in $\Omega$ with zero Dirichlet boundary conditions. It is well known that $\lambda_{1}(\Omega)$ is decreasing with respect to the inclusion of subsets. In particular $R\mapsto\lambda_{1}(B_{R}(x))$ is decreasing and $\lambda_{1}(B_{R}(x))\to \lambda_{1}(M)$ as $R\to\infty$.


\

For any $x\in M$, for any $s>0$ and for any $0\leq a < b\leq +\infty$, we define
\begin{align*}
\mathcal{L}_{x}(s) &:= \{y \in M\,:\,G(x,y)=s \},\\
\mathcal{L}_{x}(a,b)&:= \{y \in M\,:\, a< G(x,y)< b \}.
\end{align*}

\

\section{Local gradient estimate for harmonic functions} \label{sec-grad}

In this section we improve \cite[Lemma 3.1]{cmp}. We set

$$
k_R(z):=\sup_{B_R(z)}\frac{\varphi''(r(y))}{\varphi(r(y))}
$$
for $z\in M$ and $R>0$;
$$
i_R(z):=\begin{cases}
\sqrt{k_R}\coth\left(\sqrt{k_R(z)} R/2\right)&\text{if }\,k_R(z)>0 \\
\frac{2}{R} &\text{if }\,k_R(z)=0.
\end{cases}
$$
\begin{lemma}\label{lemma00} Let $R>0$ and $z\in M$. Let $u\in C^{2}(B_{R}(z))$ be a positive harmonic function in $B_{R}(z)$. Then
$$
|\nabla u(\xi)| \leq C \sqrt{\max\left\{k_R(z), \frac{i_R(z)}{R}, \frac{1}{R^2}\right\}}\, u(\xi)\quad\text{for any}\quad \xi\in B_{R/2}(z),
$$
for some positive constant $C>0$.
\end{lemma}
\begin{proof}
Following the classical argument of Yau, let $v:=\log u$. Then
$$
\Delta v = - |\nabla v|^{2} .
$$
Let $\eta(\xi)=\eta(d(\xi))$, with $d(\xi):=\operatorname{dist}(\xi,z)$, a smooth cutoff function such that $\eta(\xi)\equiv 1$ on $B_{R/2}(z)$, with support in $B_{R}(z)$, $0\leq \eta\leq 1$ and
$$-\frac{4}{R}\leq \frac{\eta'}{\eta^{1/2}} \leq 0 \quad\text{and}\quad \frac{|\eta''|}{\eta} \leq \frac{8}{R^{2}}.$$ Let $w=\eta^{2}|\nabla v|^{2}$. Then
\begin{align*}
\frac12 \Delta w &= \frac12 \eta^{2} \Delta |\nabla v|^{2} + \frac12 |\nabla v|^{2} \Delta \eta^{2} + \langle \nabla |\nabla v|^{2},\nabla \eta^{2}\rangle.
\end{align*}
Then, from classical Bochner-Weitzenb\"och formula and Newton inequality, one has
\begin{align*}
\frac12 \Delta |\nabla v|^{2} & =  |\nabla^{2} v|^{2} + \ricc(\nabla v,\nabla v) + \langle \nabla v,\nabla \Delta v\rangle \\
&\geq \frac1n (\Delta v)^{2} - (n-1) \frac{\varphi''}{\varphi} |\nabla v|^{2} - \langle \nabla |\nabla v|^{2},\nabla v\rangle \\
&= \frac1n |\nabla v|^{4} - (n-1) \frac{\varphi''}{\varphi} |\nabla v|^{2} - \langle \nabla |\nabla v|^{2},\nabla v\rangle .
\end{align*}
Moreover, by Laplacian comparison, since $\ricc\geq -(n-1)k_R(z)$ in $B_R(z)$, we have
\begin{align*}
\frac12 \Delta \eta^{2} &= \eta \eta' \Delta \rho + \eta \eta'' + (\eta')^{2} \\
&\geq (n-1)i_R(z)\eta\eta' + \eta \eta''+ (\eta')^{2}\\
&\geq -\frac{4}{R} \left((n-1)i_R(z)+\frac{2}{R}\right)\eta
\end{align*}
pointwise in $B_{R}(z)\setminus (\{z\}\cup \operatorname{Cut}(z))$ and weakly on $B_{R}(z)$. Thus,
\begin{align*}
\frac12 \Delta w &\geq \frac1n \frac{w^{2}}{\eta^{2}}-(n-1)\frac{\varphi''}{\varphi}w - \frac{4}{R}\left((n-1)i_R(z)+\frac{2}{R}\right)\frac{w}{\eta} \\
&-4\frac{|\eta'|^{2}}{\eta^{2}}w + \frac{2}{\eta}\langle \nabla w,\nabla \eta\rangle-\langle \nabla w,\nabla v\rangle + \frac{2}{\eta}\langle \nabla v,\nabla \eta \rangle w \\
&\geq \frac1n \frac{w^{2}}{\eta^{2}}-(n-1)\frac{\varphi''}{\varphi}w - \frac{4}{R}\left((n-1)i_R(z)+\frac{2}{R}\right)\frac{w}{\eta} \\
& + \frac{2}{\eta}\langle \nabla w,\nabla \eta\rangle-\langle \nabla w,\nabla v\rangle - \frac{64}{R^{2}}\frac{ w}{\eta}-\frac{8}{R}\frac{ w^{3/2}}{\eta^{3/2}}\\
&\geq \frac{1}{2n} \frac{w^{2}}{\eta^{2}}-(n-1)\frac{\varphi''}{\varphi}w - \frac{4}{R}\left((n-1)i_R(z)+\frac{18+8n}{R}\right)\frac{w}{\eta} \\
& + \frac{2}{\eta}\langle \nabla w,\nabla \eta\rangle-\langle \nabla w,\nabla v\rangle.
\end{align*}
Let $q$ be a maximum point of $w$ in $\overline{B}_{R}(z)$. Since $w\equiv0$ on $\partial B_{R}(z)$, we have  $q\in B_{R}(z)$. First assume $q\notin \operatorname{Cut}(z)$. At $q$, we obtain
\begin{align*}
0 &\geq \left[\frac{1}{2n} w - (n-1)\frac{\varphi''}{\varphi}-\frac{4}{R}\Big((n-1)i_R(z)+\frac{18+8n}{R}\Big)\right]w.
\end{align*}
So
$$
w(q)\leq 2n(n-1)\frac{\varphi''\big(r(q)\big)}{\varphi\big(r(q)\big)}+\frac{8n(n-1)}{R}i_R(z)+\frac{144n+64n^2}{R^2}.
$$
Thus, for any $\xi \in B_{R/2}(z)$,
\begin{align*}
|\nabla v(\xi)|^{2}&\leq 2n(n-1)\frac{\varphi''\big(r(q)\big)}{\varphi\big(r(q)\big)}+\frac{8n(n-1)}{R}i_R(z)+\frac{144n+64n^2}{R^2}\\
&\leq 2n(n-1)k_R(z)+\frac{8n(n-1)}{R}i_R(z)+\frac{144n+64n^2}{R^2}
\end{align*}
We get
$$
\frac{|\nabla u(\xi)|}{u(\xi)}=|\nabla v(\xi)| \leq C \sqrt{\max\left\{k_R(z), \frac{i_R(z)}{R}, \frac{1}{R^2}\right\}}.
$$
 for some positive constant $C>0$. By standard Calabi trick (see \cite{cal, cy}), the same estimate can be obtained when $q\in \operatorname{Cut}(z)$. This concludes the proof of the lemma.

\end{proof}

As a corollary we have the following

\begin{cor}\label{lemma0} Let $(M,g)$ be non-parabolic. If $r(z)>R>0$, then
$$
|\nabla G(p,z)| \leq C \sqrt{Q_{R}(z)}\, G(p,z),
$$
for some positive constant $C>0$.
\end{cor}

\

\section{Green's function estimates} \label{sec-est}

\subsection{Pointwise estimate}

\begin{lemma}\label{lemma1} Let $(M,g)$ be non-parabolic and let $a>0$ and $y\in M\setminus B_{a}(p)$. Then
$$
A^{-1} \exp \left(-B\, \omega(y)\right) \leq G(p,y) \leq A \exp \left(B\,  \omega(y)\right),
$$
with $A:=\max\{ \max_{\partial B_a(p)}G(p,\cdot), \left(\min_{\partial B_a(p)}G(p,\cdot)\right)^{-1}\}$ and $B=2n(n-1)$.
\end{lemma}
\begin{proof} Let $y\in M\setminus \overline{B_{a}(p)}$ with $a> 0$ and consider the minimal geodesic $\gamma$ joining $p$ to $y$ and let $y_{0}\in\partial B_{a}(p)$ be a point of intersection of $\gamma$ with $\partial B_{a}(p)$. Since $G(p,\cdot)$ is harmonic in $B_{r(z)/4}(z)$, for every $z\in \gamma$ with $r(z)\geq a$, by Lemma \ref{lemma0} we get
$$
|\nabla G(p,z)| \leq C \sqrt{Q_{r(z)/4}(z)}\,G(p,z) .
$$
We have
\begin{align*}
G(p,y)&=G(p,y_0)+\int_{a}^{r(y)}\langle \nabla G(p,\gamma(s)), \dot{\gamma}(s)\rangle \,ds \\
&\leq G(p,y_0) + C\int_{a}^{r(y)} \sqrt{Q_{\frac{r(\gamma(s))}{4}}\big(r(\gamma(s))\big)} G(p,\gamma(s)) \,ds.
\end{align*}
By Gronwall inequality,
$$
G(p,y) \leq G(p,y_0) \exp\left(C\int_{a}^{r(y)} \sqrt{Q_{\frac{r(\gamma(s))}{4}}\big(r(\gamma(s))\big)}\,ds\right)\leq A \exp\left(B\,\omega(y)\right),
$$
with $A:=\max\{ \max_{\partial B_a(p)}G(p,\cdot), \left(\min_{\partial B_a(p)}G(p,\cdot)\right)^{-1}\}$ and $B=2n(n-1)$.
Similarly,
$$
G(p,y) \geq A^{-1} \exp\left(-B\,\omega(y)\right).
$$
\end{proof}

%
%

\begin{rem} \label{remark101} We also note that
$$
\mathcal{L}_{p}\left(A \exp \left(B\, \omega(a)\right),\infty\right) \subset B_{a}(p).
$$
In fact, let $y\in M\setminus B_a(p)$ and take $j>r(y)$. Since $G_{j}(p,y)\leq G(p,y)$ and $G_{j}(p,\cdot)\equiv 0$ on $\partial B_{j}(p)$, by Lemma \ref{lemma1}, we have
$$
G_{j}(p,y)\leq A \exp \left(B \omega(a)\right)\quad\text{on}\quad \partial\left(B_{j}(p)\setminus B_{a}(p)\right);
$$
note that the right hand side is independent of $y$. Since $y\mapsto G_{j}(x,y)$ is harmonic in $B_{j}(p)\setminus B_{a}(p)$, by maximum principle,
$$
G_{j}(p,y)\leq A \exp \left(B  \omega(a)\right)\quad\text{in}\quad B_{j}(p)\setminus B_{a}(p).
$$
Sending $j\to\infty$, by \eqref{eq9}, we obtain
$$
G(p,y)\leq A \exp \left(B  \omega(a)\right)\quad\text{in}\quad M\setminus B_{a}(p),
$$
and the claim follows.

\end{rem}


\subsection{Auxiliary estimates}

\begin{lemma}\label{lemma2} Let $(M,g)$ be non-parabolic. For any $s>0$, there holds
$$
\int_{\mathcal{L}_{p}(s)}|\nabla G(p,y)|\,dA(y) = 1
$$
where $dA(y)$ is the $(n-1)$-dimensional Hausdorff measure on $\mathcal{L}_{x}(s)$. As a consequence, by the co-area formula, for any $0<a<b$, there holds
$$
\int_{\mathcal{L}_{p}(a,b)}\frac{|\nabla G(p,y)|^2}{G(p,y)}\,dy = \log\left(\frac{b}{a}\right) \,.
$$
\end{lemma}
For the proof see \cite{ms}. Moreover, we get the following weighted integrability property for the Green's function.

\begin{lemma}\label{lemmastoc} Assume that $(M,g)$ satisfies the property $\left(\mathcal{P}^\infty_{\rho_R}\right)$. Fix $m\geq R_0$. Then, for any $R_1>0$ such that $B_m(p)\subset B^{\rho_m}_{R_1}(p)$, one has
$$
\int_{M\setminus B^{\rho_m}_{2R_1}(p)} \rho_m(y)\,|G(p,y)|^2\,dy < \infty \,.
$$
\end{lemma}
\begin{rem} Note that $B_m(p)\subset B^{\rho_m}_{R_1}(p)$ for every $R_1$ large enough. \end{rem}
\begin{proof} In order to simplify the notation, let $\rho\equiv \rho_m$. Fix $R_1>0$ such that $B_m(p)\subset B^\rho_{R_1}(p)$ and let  $\phi$ be defined as
\begin{equation*}
\phi(x):=\begin{cases}
0 &  \textrm{on } B^\rho_{R_1}(p) \\
\frac{r_\rho(x)-R_1}{R_1} & \textrm{on } B^\rho_{2R_1}(p)\setminus B^\rho_{R_1}(p)\\
1 &  \textrm{on } M\setminus B^\rho_{2R_1}(p) \,.
\end{cases}
\end{equation*}
Let $R>2R_1$ and $G^{\rho}_{R}(p,y)$ be the Green's function of $-\Delta$ in $B^{\rho}_{R}(p)$ satisfying zero Dirichlet boundary conditions on $\partial B^{\rho}_{R}(p)$. Following the proof in \cite{liwa1}, since $G^{\rho}_R$ is harmonic in $B^{\rho}_{R}(p)$, one has
\begin{align*}
\int_{B^{\rho}_R(p)}|\nabla \left(\phi \,G^{\rho}_R\right)|^2\,dV &= \int_{B^{\rho}_R(p)}|\nabla \phi|^2 \left(G^{\rho}_R\right)^2\,dV + \int_{B^{\rho}_R(p)}|\nabla G^{\rho}_R |^2 \phi^2\,dV\\
&+ 2 \int_{B^{\rho}_R(p)}\langle \nabla \phi, \nabla G^{\rho}_R \rangle \phi G^{\rho}_R  \,dV  \\
&= \int_{B^{\rho}_R(p)}|\nabla \phi|^2 \left(G^{\rho}_R\right)^2\,dV + \frac{1}{2}\int_{B^{\rho}_R(p)}\Delta\left(G^{\rho}_R\right)^2 \phi^2\,dV\\
&+ 2 \int_{B^{\rho}_R(p)}\langle \nabla \phi, \nabla G^{\rho}_R \rangle \phi G^{\rho}_R  \,dV \\
&= \int_{B^{\rho}_R(p)}|\nabla \phi|^2 \left(G^{\rho}_R\right)^2\,dV
\end{align*}
where the last equality follows by integration by parts and the fact that $G^{\rho}_{R}(p,y)$ vanishes on $\partial B^{\rho}_{R}(p)$. Hence, the weighted Poincar\'e inequality yields
\begin{align*}
\int_{M\setminus B^\rho_{R_1}(p)} \rho \,\left(G^{\rho}_R\right)^2\phi^2\,dV \leq \int_{B^{\rho}_R(p)}|\nabla \left(\phi \,G^{\rho}_R\right)|^2\,dV \leq \frac{1}{R_1^2}\int_{B^{\rho}_{2R_1}(p)\setminus B^{\rho}_{R_1}(p)}\rho\,\left(G^{\rho}_R\right)^2\,dV
\end{align*}
Letting $R\rightarrow \infty$, by Fatou's lemma and uniform convergence of $G_R^\rho \rightarrow G$ on compact subsets, we get
$$
\int_{M\setminus B^\rho_{2R_1}(p)} \rho \,G^2\,dV \leq \frac{1}{R_1^2}\int_{B^{\rho}_{2R_1}(p)\setminus B^{\rho}_{R_1}(p)}\rho\, G^{2}\,dV
$$
and the thesis follows.
\end{proof}

We expect a decay estimate similar to the one in \cite[Theorem 2.1]{liwa1}. However we leave out this refinement since it is not necessary in our arguments.

\subsection{Integral estimates on level sets}

We begin by noting that, using Remark \ref{remark101} and the fact that $G(p,\cdot)\in L^1_{\text{loc}}(M)$ one has the following integral estimate on large level sets.
\begin{proposition}\label{lemma3} Let $(M,g)$ be non-parabolic. Choose $A,B$ as in Lemma \ref{lemma1}. Then
\begin{align*}
\int_{\mathcal{L}_{p}\left(A \exp \left(B\, \omega(a)\right),\infty\right)} &G(p,y)\,dy <\infty.
\end{align*}
\end{proposition}

For intermediate levels sets, we get the following key inequality.

\begin{proposition}\label{claim2} Assume that $(M,g)$ satisfies the property $\left(\mathcal{P}^\infty_{\rho_R}\right)$. Then, there exists a positive constant $C$ such that, for any function $f$ and any $0<\delta<1$, $\varepsilon >0$ satisfying $\mathcal{L}_p \left(\frac{\delta\varepsilon}{2},2\varepsilon\right) \subset M \setminus B_m(p)$ for some $m>R_0$, one has
$$
\left|\int_{\mathcal{L}_{p}(\delta \varepsilon, \varepsilon)} G(p,y)\,f(y)\,dy \right| \leq C \left(-\log\delta +1\right) \sup_{\mathcal{L}_{p}(\delta \varepsilon, \varepsilon)} \left|\frac{f}{\rho_m}\right|\,.
$$
\end{proposition}
\begin{proof}
We follow the general argument in \cite{liwa1} and \cite{ms}; however some relevant differences are in order, due to the use of the property $\left(\mathcal{P}^\infty_{\rho_R}\right)$. Let $\phi:=\chi \psi$ with
\begin{equation*}
\chi(y):=\begin{cases}
\frac{1}{\log 2} \log \left(\frac{2 G(p,y)}{\delta \epsilon}\right) & \textrm{on }  \mathcal{L}_p \left(\frac{\delta\varepsilon}{2},\delta\varepsilon\right)\\
1 & \textrm{on } \mathcal{L}_p \left(\delta\varepsilon,\varepsilon\right)\\
\frac{1}{\log 2} \log \left(\frac{2 \varepsilon}{G(p,y)}\right) & \textrm{on }  \mathcal{L}_p \left(\varepsilon,2\varepsilon\right) \\
0 & \textrm{elsewhere}
\end{cases}
\end{equation*}
and for any fixed $R>0$
\begin{equation*}
\psi(y):=\begin{cases}
1 &  \textrm{on } B^{\rho_m}_{R}(p) \\
R+1-r_{\rho_m}(y) & \textrm{on } B^{\rho_m}_{R+1}(p)\setminus B^{\rho_m}_{R}(p)\\
0 &  \textrm{on } M\setminus B^{\rho_m}_{R+1}(p) \,.
\end{cases}
\end{equation*}
By the weighted Poincar\'e inequality at infinity we get
\begin{align*}
\left|\int_{\mathcal{L}_{p}(\delta \varepsilon, \varepsilon)\cap B^{\rho_m}_{R}(p)} G(p,y)\,f(y)\,dy \right| &\leq \int_{\mathcal{L}_{p}(\delta \varepsilon, \varepsilon)\cap B^{\rho_m}_{R}(p)} G(p,y)\,|f(y)|\,dy \\
&\leq \sup_{\mathcal{L}_{p}(\delta \varepsilon, \varepsilon)\cap B^{\rho_m}_{R}(p)} \left|\frac{f}{\rho_m}\right| \, \int_{M} \rho_m(y)\,G(p,y) \phi^2(y)\,dy \\
&\leq \sup_{\mathcal{L}_{p}(\delta \varepsilon, \varepsilon)\cap B^{\rho_m}_{R}(p)} \left|\frac{f}{\rho_m}\right| \, \int_{M} \left|\nabla \left( \sqrt{G(p,y)} \phi(y)\right)\right|^2\,dy \,.
\end{align*}
We estimate
\begin{align*}
\int_{M} \left|\nabla \left( \sqrt{G(p,y)} \phi(y)\right)\right|^2\,dy  &\leq \frac{1}{2} \int_{\mathcal{L}_{p}(\frac{\delta \varepsilon}{2}, 2\varepsilon)} \frac{|\nabla G(p,y)|^2}{G(p,y)}\,dy + 2 \int_M G(p,y)|\nabla \phi|^2 \,dy \\
&= C(-\log\delta+1) + 2 \int_M G(p,y)|\nabla \phi|^2 \,dy
\end{align*}
where we used Lemma \ref{lemma2} in the last equality. On the other hand
\begin{align*}
\int_M G(p,y)|\nabla \phi|^2 \,dy &\leq 2 \int_M G(p,y)|\nabla \chi|^2 \psi^2 \,dy + 2 \int_M G(p,y)|\nabla \psi |^2 \chi^2 \,dy \\
&\leq 2(\log 2)^2 \int_{\mathcal{L}_{p}(\frac{\delta \varepsilon}{2}, 2\varepsilon)} \frac{|\nabla G(p,y)|^2}{G(p,y)}\,dy \\
 &\quad\ + 2 \int_{B^\rho_{R+1}(p)\setminus B^\rho_{R}(p)} \rho_m(y) \,G(p,y) \chi^2 \,dy \\
&\leq C(-\log\delta+1)+ \frac{4}{\delta\varepsilon} \int_{B^{\rho_m}_{R+1}(p)\setminus B^{\rho_m}_{R}(p)} \rho_m(y) \,G^2(p,y) \,dy \,.
\end{align*}
Now we let $R\rightarrow \infty$ and use Lemma \ref{lemmastoc}. The thesis now follows.
\end{proof}

In the special case when $M$ is non-parabolic with positive minimal Green's function $G$ and with weight $\rho(x)=\frac{|\nabla G(p,x)|^2}{4 G^2(p,x)}$, we have the following refinement of Proposition \ref{claim2}.

\begin{proposition}\label{claim3} Assume that $(M,g)$ is non-parabolic with positive minimal Green's function $G$ and with weight $\rho(x)=\frac{|\nabla G(p,x)|^2}{4 G^2(p,x)}$. Then there exists a positive constant $C$ such that for any function $f$ and any $0<\delta<1$, $\varepsilon >0$ one has
$$
\left|\int_{\mathcal{L}_{p}(\delta \varepsilon, \varepsilon)} G(p,y)\,f(y)\,dy \right| \leq C \left(-\log\delta \right) \sup_{\mathcal{L}_{p}(\delta \varepsilon, \varepsilon)} \left|\frac{f}{\rho}\right|\,.
$$
\end{proposition}

\begin{proof}
  We have
  \begin{align*}
  \left|\int_{\mathcal{L}_{p}(\delta \varepsilon, \varepsilon)} G(p,y)\,f(y)\,dy \right| &\leq\sup_{\mathcal{L}_{p}(\delta \varepsilon, \varepsilon)} \left|\frac{f}{\rho}\right| \left(\int_{\mathcal{L}_{p}(\delta \varepsilon, \varepsilon)} G(p,y)\,\rho(y)\,dy\right)\\
  &=\frac{1}{4}\sup_{\mathcal{L}_{p}(\delta \varepsilon, \varepsilon)} \left|\frac{f}{\rho}\right| \left(\int_{\mathcal{L}_{p}(\delta \varepsilon, \varepsilon)} \frac{|\nabla G(p,y)|^2}{G(p,y)}\,dy\right)\\
   &=\frac{1}{4}\left(-\log\delta \right)\sup_{\mathcal{L}_{p}(\delta \varepsilon, \varepsilon)} \left|\frac{f}{\rho}\right|,
  \end{align*}
  where we have used Lemma \ref{lemma2} in the last equality.
\end{proof}

\

\section{Proof of Theorem \ref{teo1}} \label{sec-proofs}

In order to prove Theorem \ref{teo1}, we will show that
$$
|u(x)|=\left| \int_{M}G(x,y)f(y)\,dy \right| \leq v(x),
$$
with $v\in C^{0}(M)$. We divide the proof in two parts, we first consider the case when $(M,g)$ is non-parabolic and then the case when it is parabolic.

\begin{proof}[Proof of Theorem \ref{teo1}] {\bf Case 1:} {\em  $(M,g)$ non-parabolic.}

\

By assumption, $(M,g)$ satisfies $\left(\mathcal{P}_{\rho_R}^\infty\right)$. Let $x\in M$ and choose $R=R(x)>R_0$ large enough so that $x\in B_R (p)$. One has

\begin{align*}
\left|\int_M G(x,y)\,f(y)\, dy\right| &\leq \left| \int_{B_R(p)} G(x,y)\,f(y)\,dy \right|+\left|\int_{M\setminus B_R(p)} G(x,y)\,f(y)\,dy\right|\\
&\leq C_1(x) + \int_{M\setminus B_R(p)} G(x,y)\,|f(y)|\,dy
\end{align*}
since $G(x,\cdot)\in L^1_{\text{loc}}(M)$.  Hence, by Harnack's inequality, we have
\begin{align}\label{eq501}
\left|\int_M G(x,y)\,f(y)\, dy\right| &\leq C_1(x) + C_2(x)\int_{M\setminus B_R(p)} G(p,y)\,|f(y)|\,dy \\ \nonumber
&\leq C_1(x) + C_2(x)\int_{M} G(p,y)\,|f(y)|\,dy \,,
\end{align}
where $C_2(x)$ can be chosen as the constant in the Harnack's inequality for the ball $B_{r(x)+1}(p)$. Then we estimate

\begin{align*}
 \int_{M}G(p,y)\,|f(y)|\,dy   &= \int_{\mathcal{L}_{p}\left(0, \,A \exp \left(B\, \omega(a)\right)\right)} G(p,y)\,|f(y)|\,dy  \\
 &\,\,\,+  \int_{\mathcal{L}_{p}\left(A \exp \left(B\, \omega(a)\right),\infty\right)} G(p,y)\,|f(y)|\,dy \,.
\end{align*}
By Proposition \ref{lemma3}, Remark \ref{remark101}  we get
\begin{align}\label{eq128}
 \int_{M}G(p,y)\,|f(y)|\,dy   &\leq \int_{\mathcal{L}_{p}\left(0, A \exp \left(B\, \omega(a)\right)\right)} G(p,y)\,|f(y)|\,dy  + C_3(a)
\end{align}
for some positive constant $C_3(a)$. To estimate the first integral, we observe that, for any $m_{0}=m_{0}(x)\geq a$ one has
\begin{align}\label{eq129}
 \int_{\mathcal{L}_{p}\left(0, \,A \exp \left(B\, \omega(a)\right)\right)} &G(x,y)\,|f(y)|\,dy  =  \int_{\mathcal{L}_{p}\left(0, \,(2A)^{-1}\exp(-B\omega(m_{0}))\right)}G(x,y)\,|f(y)|\,dy \nonumber \\
&\quad+ \int_{\mathcal{L}_{p}\left((2A)^{-1}\exp(-B\omega(m_{0})),\,A \exp \left(B\, \omega(a)\right)\right)}G(x,y)\,|f(y)|\,dy \,.
\end{align}
We need the following lemma.
\begin{lemma}\label{lemma5} Choose $A,B$ as in Lemma \ref{lemma1}. For any $m\geq m_0\geq a$ one has
\begin{equation}\label{eq400}\mathcal{L}_{p}\left(0, A^{-1}\exp(-B\omega(m))\right) \subset M \setminus B_m(p).
\end{equation}
\end{lemma}
\begin{proof}
Since $m_{0}\geq a$, by Remark \ref{remark101} imply
\begin{equation}\label{eq300}
\mathcal{L}_{p}\left(0, A^{-1}\exp(-B\omega(m_{0}))\right) \subset\mathcal{L}_{p}\left(0, A^{-1} \exp \left(-B\, \omega(a)\right)\right)\subset M\setminus B_{a}(p) .
\end{equation}
If
$$
z\in \mathcal{L}_{p}\left(0, A^{-1}\exp(-B\omega(m))\right) \subset M\setminus B_{a}(p) \,,
$$ then by Lemma \ref{lemma1}
$$
A^{-1}\exp(-B\omega(m)) \geq G(p,z) \geq A^{-1}\exp(-B\omega(z)) \,.
$$
Thus,
$$
\omega(z)\geq \omega(m)
$$
and, by monotonicity of $\omega$, we obtain $r(z)\geq m$.
\end{proof}
In particular, we get
$$
\mathcal{L}_{p}\left(0, (2A)^{-1}\exp(-B\omega(m_{0}))\right) \subset \mathcal{L}_{p}\left(0, A^{-1}\exp(-B\omega(m_{0}))\right) \subset M\setminus B_{m_0}(p).
$$
Thus,
$$
\mathcal{L}_{p}\left((2A)^{-1}\exp(-B\omega(m_{0})),\,A \exp \left(B\, \omega(a)\right)\right) \subset B_{m_{0}}(p)
$$
Then,  since $G(x,\cdot)\in L^1_{\text{loc}}(M)$, we get
\begin{align}\label{eq130}
\int_{\mathcal{L}_{p}\left((2A)^{-1}\exp(-B\omega(m_{0})),\,A \exp \left(B\, \omega(a)\right)\right)}G(x,y)\,|f(y)|\,dy \leq C_4(a,m_0).
\end{align}
Now, for any $m\geq m_{0}$, let
\begin{equation}\label{11}
  \varepsilon:=(2A)^{-1}\exp(-B\omega(m)),\quad\quad\delta:=\exp(B\omega(m)-B\omega(m+1)).
\end{equation}
By Lemma \ref{lemma5},
$$
\mathcal{L}_p(0,2\varepsilon) \subset M\setminus B_m(p).
$$
Hence we can apply Proposition \ref{claim2} obtaining
\begin{align}\label{eq201}
&\int_{\mathcal{L}_{p}\left(0, (2A)^{-1}\exp(-B\omega(m_{0}))\right)}G(x,y)\,|f(y)|\,dy  \\\nonumber
&= \sum_{m\geq m_{0}} \int_{\mathcal{L}_{p}\left((2A)^{-1}\exp(-B\omega(m+1)), (2A)^{-1}\exp(-B\omega(m))\right)}G(x,y)\,|f(y)|\,dy  \\\nonumber
&\leq C \sum_{m\geq m_{0}}^{\infty}\left(\omega(m+1)-\omega(m)+1\right)\sup_{\mathcal{L}_{p}\left((2A)^{-1}\exp(-B\omega(m+1)), (2A)^{-1}\exp(-B\omega(m))\right)}\left|\frac{f}{\rho_m}\right|\\ \nonumber
&\leq C \sum_{m\geq m_{0}}^{\infty}\left(\omega(m+1)-\omega(m)+1\right)\sup_{\mathcal{L}_{p}\left(0, A^{-1}\exp(-B\omega(m))\right)}\left|\frac{f}{\rho_m}\right|\\ \nonumber
&\leq C \sum_{m\geq m_{0}}^{\infty}\left(\omega(m+1)-\omega(m)+1\right)\sup_{M \setminus B_m(p)}\left|\frac{f}{\rho_m}\right| <\infty \,,
\end{align}
where in the last inequality we used Lemma \ref{lemma5}. The proof of Theorem \ref{teo1} is complete in this case.

\

\noindent {\bf Case 2:} {\em  $(M,g)$ parabolic.}

\

Let $G(x,y)$ be a Green's function on $M$ (which is positive inside a certain ball, and negative outside). Fix any $R>0$ and let $\rho	\equiv\rho_{R_0}$. Note that, arguing as in the proof of \eqref{eq501}, it is sufficient to estimate
\begin{align*}
 \int_{M}|G(p,y)||f(y)|\,dy &= \int_{M\setminus B^\rho_{R}(p)}|G(p,y)||f(y)|\,dy + \int_{B^\rho_{R}(p)}|G(p,y)||f(y)|\,dy \\
&\leq \int_{M\setminus B^\rho_{R}(p)}|G(p,y)||f(y)|\,dy+  C,
\end{align*}
since $G(p,\cdot)\in L^{1}_{\rm{loc}}(M)$ and $f$ is locally bounded. We have that
$$
M\setminus B^\rho_{R}(p) = \bigcup_{i=1}^{N} E_{i},
$$
where each $E_{i}$ is an end with respect to $B^\rho_{R}(p)$. Note that every end $E_{i}$ is parabolic. In fact, if at least one end $E_{i}$ is non-parabolic, then $(M,g)$ is non-parabolic (see \cite{li} for a nice overview), but we are in the case that $(M,g)$ is parabolic. Since every $E_{i}$ is parabolic, every $E_{i}$ has finite weighted volume (see \cite{liwa2}), i.e.
$$
\int_{E_i} \rho\,dy < \infty \,.
$$
Now choose $R$ large enough so that we can apply Lemma \ref{lemmastoc} obtaining
\begin{align*}
&\int_{M\setminus B^\rho_{R}(p)}|G(p,y)||f(y)|\,dy \\
&\qquad\leq \left(\int_{M\setminus B^\rho_{R}(p)}\rho (y)|G(p,y)|^{2}\,dy\right)^{\frac{1}{2}}\left(\int_{M\setminus B^\rho_{R}(p)}\rho(y)\left(\frac{|f(y)|}{\rho(y)}\right)^{2}\,dy\right)^{\frac{1}{2}}\\
&\qquad\leq C\, \sup_{M\setminus B_{R_{0}}(p)} \left|\frac{f}{\rho}\right| \int_{M\setminus B^\rho_{R}(p)}\rho\,dy <\infty\,.
\end{align*}
This concludes the proof of Theorem \ref{teo1}.

\end{proof}

\begin{proof}[Proof of Theorem \ref{teo2}] We start as in the proof of Theorem \ref{teo1} using \eqref{eq501}, \eqref{eq128}, \eqref{eq129} and \eqref{eq130}. Then, similar to \eqref{eq201}, using Proposition \ref{claim3}, we obtain 
\begin{align*}
&\int_{\mathcal{L}_{p}\left(0, (2A)^{-1}\exp(-B\omega(m_{0}))\right)}G(x,y)\,|f(y)|\,dy  \\\nonumber
&= \sum_{m\geq m_{0}} \int_{\mathcal{L}_{p}\left((2A)^{-1}\exp(-B\omega(m+1)), (2A)^{-1}\exp(-B\omega(m))\right)}G(x,y)\,|f(y)|\,dy  \\
&\leq C \sum_{m\geq m_{0}}^{\infty}\left(\omega(m+1)-\omega(m)\right)\sup_{M \setminus B_m(p)}\left|\frac{f}{\rho}\right| <\infty \,,
\end{align*}
Then
$$
\left| \int_{M}G(x,y)f(y)\,dy \right| <\infty
$$
and the proof of Theorem \ref{teo2} is complete.
\end{proof}
\

\section{Cartan-Hadamard and model manifolds} \label{sec-ex}

We consider Cartan-Hadamard manifolds, i.e.~complete, non-compact, simply connected Riemannian manifolds with non-positive
sectional curvatures everywhere. Observe that on Cartan-Hadamard manifolds the cut locus of any point $p$ is empty.
Hence, for any $x\in M\setminus \{p\}$ one can define its polar coordinates with pole at $p$, namely $r(x) = \operatorname{dist}(x, p)$ and $\theta\in \mathbb S^{n-1}$. We have
\begin{equation*}
\textrm{meas}\big(\partial B_{r}(p)\big)\,=\, \int_{\mathbb S^{n-1}}A(r, \theta) \, d\theta^1d \theta^2 \ldots d\theta^{n-1}\,,
\end{equation*}
for a specific positive function $A$ which is related to the metric tensor \cite[Sect. 3]{gri1}. Moreover, it is direct to see that the Laplace-Beltrami operator in polar coordinates has the form
\begin{equation*}
\Delta \,=\, \frac{\partial^2}{\partial r^2} + m(r, \theta) \, \frac{\partial}{\partial r} + \Delta_{\theta} \, ,
\end{equation*}
where $m(r, \theta):=\frac{\partial }{\partial r}(\log A)$ and $ \Delta_{\theta} $ is the Laplace-Beltrami operator on $\partial B_{r}(p)$.  We have
$$
m(r,\theta) =\Delta r(x).
$$

Let $$\mathcal A:=\left\{f\in C^\infty((0,\infty))\cap C^1([0,\infty)): \, f'(0)=1, \, f(0)=0, \, f>0 \ \textrm{in}\;\, (0,\infty)\right\} .$$ We say that $(M,g)$ is a rotationally symmetric manifold or a model manifold if the Riemannian metric is given by
\begin{equation*}\label{e2}
g \,=\, dr^2+\varphi(r)^2 \, d\theta^2,
\end{equation*}
where $d\theta^2$ is the standard metric on $\mathbb S^{n-1}$ and $\varphi\in \mathcal A$. In this case,
\begin{equation*}
\Delta \,=\, \frac{\partial^2}{\partial r^2} + (n-1) \, \frac{\varphi'}{\varphi} \, \frac{\partial}{\partial r} + \frac1{\varphi^2} \, \Delta_{\mathbb S^{n-1}} \, .
\end{equation*}
Note that $\varphi(r)=r$ corresponds to $M=\mathbb R^n$, while $\varphi(r)=\sinh r$ corresponds to $ M=\mathbb H^n $, namely the $n$-dimensional hyperbolic space. The Ricci curvature in the radial direction is given by
$$
\ricc( \nabla r, \nabla r) (x) = -(n-1)\frac{\varphi''(r(x))}{\varphi(r(x))}.
$$

\subsection{Cartan-Hadamard manifolds} Concerning the validity of the property $\left(\mathcal{P}_\rho^\infty\right)$ on a Cartan-Hadamard manifold we have the following result.
\begin{lemma}\label{lemma-peso}
Let $(M,g)$ be a Cartan-Hadamard manifold with
$$
\ricc( \nabla r, \nabla r) (x)\leq -C\big(1+r(x)\big)^{\gamma}
$$
for some $\gamma\in \RR$, $C>0$ and any $x\in M\setminus\{p\}$. Then $(M,g)$ satisfies the property $\left(\mathcal{P}_{\rho_{R}}^\infty\right)$ with
$$
\rho_R(x) = \begin{cases} 
C'\, r(x)^{\gamma} &\quad\hbox{if } \gamma\geq -2  \\
C'\, r(x)^{-2} &\quad\hbox{if } \gamma < -2 \end{cases}
$$
for all $R>0$ large enough and some $C'>0$.
\end{lemma}
\begin{rem} As it will be clear from the proof, we have a weighted Poincar\'e inequality on $M$ if $\gamma \leq 0$ and a the weighted Poincar\'e inequality for functions with compact support in $M\setminus B_1(p)$ if $\gamma>0$.  
\end{rem}
\begin{proof} We can find $\varphi\in \mathcal{A}$ given by
\begin{equation}\label{15}
\varphi(r)=
\begin{cases}
\exp\big(B\,r^{1+\frac{\gamma}{2}}\big) &\quad\hbox{if } \gamma>-2 \\
r^\delta &\quad\hbox{if } \gamma=-2 \\
r &\quad\hbox{if } \gamma<-2
\end{cases}
\end{equation}
for $r$ large enough, $B>0$ small, $\delta=\delta(C)>1$ such that $\ricc( \nabla r, \nabla r) (x) \leq -\frac{\varphi''(r(x))}{\varphi(r(x))}$. By the Laplacian comparison in a strong form, which is valid only on Cartan-Hadamard manifolds (see \cite[Theorem 2.15]{xin}), one has
$$
\Delta r(x) \geq
\begin{cases}
C\, r(x)^{\gamma/2} &\quad\hbox{if } \gamma\geq-2 \\
C r(x)^{-1} &\quad\hbox{if } \gamma<-2 \,.
\end{cases}
$$
Suppose $\gamma\leq 0$ and let $\alpha:=\max\{\gamma,-2\}\leq 0$. For any $u\in C^\infty_c (M)$, since $|\nabla r|^2=1$, we have
\begin{align*}
&C \int_M r(y)^\alpha \,u(y)^2\,dy\\
&\qquad\leq \int_M u(y)^2 r(y)^{\alpha/2} \Delta r (y)\,dy \\
&\qquad= -2 \int_M \langle \nabla u, \nabla r\rangle u(y) r(y)^{\alpha/2}\,dy + \frac{\alpha}{2} \int_M u(y)^2 r(y)^{\alpha/2-1} |\nabla r(y)|^2\,dy \\
&\qquad\leq 2 \int_M |u(y)| |\nabla u(y)| r(y)^{\alpha/2}\,dy\\
&\qquad\leq \frac{C}{2} \int_M r(y)^\alpha \,u(y)^2\,dy + \frac{2}{C} \int_M |\nabla u(y)|^2\,dy \,.
\end{align*}
Thus
$$
\int_M r(y)^\alpha \,u(y)^2\,dy \leq \frac{4}{C^2}  \int_M |\nabla u(y)|^2\,dy
$$
and the weighted Poincar\'e inequality on $M$ follows in this case.

Suppose now $\gamma >0$. By a Barta-type argument (see e.g. \cite[Theorem 11.17]{gri2}),
\[\lambda_1(M\setminus B_R(p)) \geq [C R^{\frac{\gamma}{2}}]^2 \quad \textrm{in}\;\; M\setminus B_R(p)\,. \]
Thus, the Poincar\'e inequality reads
\begin{align}\label{poineq}
C R^\gamma \int_M u(y)^2\,dy \leq \int_M |\nabla u(y)|^2\,dy
\end{align}
for any $u$ with compact support in $M\setminus B_R(p)$. Now let $R>1$ and, for every $k\in\mathbb{N}$, define the cutoff functions 
$$
\varphi_k(x):=\begin{cases} r(x)-k+1, &r(x)\in[k-1,k)\\ k+1-r(x), &r(x)\in[k,k+1)\\ 0 &\text{otherwise}.\end{cases}
$$
Note that $|\nabla \varphi_k|\leq 1$ and for all $x\in M\setminus B_1(p)$, $\sum_k \varphi_k =1$ and $x\in \operatorname{supp}\varphi_k$ at most for two integers $k$. If $\operatorname{supp} u \subset M \setminus B_1(p)$, we have
\begin{align*}
\int_M r(y)^\gamma \,u(y)^2\,dy &= \int_M r(y)^\gamma \,\left(\sum_k \varphi_k (y) u(y)\right)^2\,dy \\
&\leq 2\sum_k \int_M r(y)^\gamma \,\varphi_k (y)^2 u(y)^2\,dy \\
&\leq C\sum_k (k-1)^\gamma \int_M \varphi_k (y)^2 u(y)^2\,dy \\
&\leq C\sum_k \int_M |\nabla\left(\varphi_k (y) u(y)\right)|^2\,dy,
\end{align*}
where in the last passage we used \eqref{poineq} with $R=k-1$. Thus
\begin{align*}
\int_M r(y)^\gamma \,u(y)^2\,dy &\leq C\sum_k \left(\int_M u(y)^2|\nabla \varphi_k (y)|^2\,dy+\int_M \varphi_k(y)^2|\nabla u(y)|^2\,dy\right)\\
&\leq C\int_M u(y)^2\,dy+C\int_M |\nabla u(y)|^2\,dy\\
&\leq C\int_M |\nabla u(y)|^2\,dy,
\end{align*}
where in the last passage we used \eqref{poineq} with $R=1$. Hence the weighted Poincar\'e inequality holds for functions with support in $M\setminus B_1(p)$.

\medskip

Finally, the completeness of the metric $g_{\rho_R}:= {\rho_R}\, g$ follows. In fact, for any curve $\eta(s)$ parametrized by arclength with $0\leq s \leq T$, the length of $\eta$ with respect tp $g_{\rho_R}$ is given by
$$
\int_\eta \sqrt{{\rho_R}}\,ds \to \infty \quad\hbox{as } T\to \infty \,.
$$
\end{proof}

\

\noindent Let us write some estimates which will be useful both in the proof of Corollary \ref{cor-2} and in the last Subsection \ref{ssu}. Choose $\varphi\in\mathcal{A}$ as in \eqref{15} with $\gamma=\gamma_1$ obtaining
$$
\frac{\varphi'(r(x))}{\varphi(r(x))}=\begin{cases} C\,r(x)^{\gamma_1/2} &\quad\hbox{if } \gamma_1\geq -2 \\ C\,r(x)^{-1} &\quad\hbox{if } \gamma_1< -2 \end{cases}
$$
and
$$
\frac{\varphi''(r(x))}{\varphi(r(x))} = \begin{cases} C\,r(x)^{\gamma_1}+C' r(x)^{\gamma_1/2-1} &\quad\hbox{if } \gamma_1\geq -2 \\ 0 &\quad\hbox{if } \gamma_1<-2\, \end{cases}
$$
for $r(x)>R>1$. A simple computation shows that, for $R=r(x)/4$, one has
$$
K_R(x) = \begin{cases} C\, r(x)^{\gamma_1/2} &\quad\hbox{if } \gamma_1\geq -2 \\ 0 &\quad\hbox{if } \gamma_1<-2\,, \end{cases}
$$
$$
\frac{I_R(x)}{R} = \begin{cases} C\, r(x)^{\gamma_1/2-1}\coth\left(C'r(x)^{\gamma_1/2+1}\right) &\quad\hbox{if } \gamma_1\geq -2 \\ \frac{2}{r(x)^2} &\quad\hbox{if } \gamma_1<-2\, \end{cases}
$$
and
$$
Q_R(x) = \begin{cases} C\, r(x)^{\gamma_1} &\quad\hbox{if } \gamma_1\geq -2 \\ \frac{2}{r(x)^2} &\quad\hbox{if } \gamma_1<-2\,. \end{cases}
$$
Thus
$$
\omega(r) = \begin{cases} C\, r^{\gamma_1/2+1} &\quad\hbox{if } \gamma_1\geq -2 \\ C \log r &\quad\hbox{if } \gamma_1<-2\,, \end{cases}
$$
and, as $m\to\infty$,
\begin{equation}\label{asdf}
\omega(m+1)-\omega(m)  \sim\begin{cases} C\, m^{\gamma_1/2} &\quad\hbox{if } \gamma_1\geq -2 \\ C m^{-1} &\quad\hbox{if } \gamma_1<-2\,. \end{cases}
\end{equation}
On the other hand, using Lemma \ref{lemma-peso} with $\gamma=\gamma_2$, we get the estimate
$$
\sup_{M\setminus B_m(p)} \frac{1}{\rho_m} \leq \begin{cases} C\,m^{-\gamma_2} &\quad\hbox{if }\gamma_2\geq -2 \\
C\, m^{2} &\quad\hbox{if } \gamma_2 < -2 \end{cases}\,.
$$

\begin{proof}[Proof of Corollary \ref{cor-2}]  For $\gamma_1\geq \gamma_2$ and $\gamma_1\geq 0$, we get
$$
\sum_{m}^{\infty}\Big(\omega(m+1)-\omega(m)+1\Big)\sup_{M\setminus B_m(p)}\left|\frac{f}{\rho_m}\right| \leq  \begin{cases} C \sum_{m}^{\infty} \,m^{\gamma_1/2-\gamma_2-\alpha} &\quad\hbox{if }\gamma_2\geq -2 \\
C \sum_{m}^{\infty}\, m^{2+\gamma_1/2-\alpha} &\quad\hbox{if } \gamma_2< -2. \end{cases}
$$
and the thesis immediately follows.
\end{proof}


\subsection{Optimality on rotationally symmetric manifolds}\label{ssu} We show that the assumptions in Theorem \ref{teo2} are sharp on model manifolds. Let $(M,g)$ be a rotationally symmetric manifold with $\varphi\in\mathcal{A}$ defined as in \eqref{15} for any $r>1$. 
%
%
One has
$$
\int_{M}G(x,y)f(y)\,dy<\infty \quad\quad\hbox{for any }\,  x \in M  \quad \Longleftrightarrow \quad \int_{M}G(p,y)f(y) \,dy<\infty .
$$
Hence a solution of $\Delta u = f$ in $M$ exists if and only if
$$
u(p)=\int_{0}^{\infty}\left(\int_{r}^{\infty}\frac{1}{\varphi(t)^{n-1}}dt\right)f(r)\,\varphi(r)^{n-1}\,dr <\infty.
$$

\noindent {\em Case 1:} $\gamma>-2$. With our choice of $\varphi$, by the change of variable $s=t^{1+\frac{\gamma}{2}}$, it is easily seen that, for any $r>0$ sufficiently large
\begin{equation}\label{asd}
\int_{r}^{\infty}\frac{1}{\varphi(t)^{n-1}}dt \sim C r^{-\frac{\gamma}{2}}\exp\left(-(n-1)r^{1+\frac{\gamma}{2}}\right).
\end{equation}
Hence
\begin{align*}
 \frac 1{C} \int_{1}^{\infty} & r^{-\frac{\gamma}{2}}\exp\left(-(n-1)r^{1+\frac{\gamma}{2}}\right) \frac{1}{\big(1+r\big)^{\alpha}}\exp\left((n-1)r^{1+\frac{\gamma}{2}}\right)\,dr  \leq |u(p)|\\&\leq C \int_{1}^{\infty} r^{-\frac{\gamma}{2}}\exp\left(-(n-1)r^{1+\frac{\gamma}{2}}\right) \frac{1}{\big(1+r\big)^{\alpha}}\exp\left((n-1)r^{1+\frac{\gamma}{2}}\right)\,dr
\end{align*}
Therefore,
\begin{align*}
\frac 1 C\int_{1}^{\infty}\frac{1}{r^{\alpha+\frac{\gamma}{2}}}\,dr &\leq |u(p)|\leq C \int_{1}^{\infty}\frac{1}{r^{\alpha+\frac{\gamma}{2}}}\,dr\,.
\end{align*}
This yields that
$$|u(p)|<\infty \quad \textrm{
if and only if} \quad
\alpha>1-\frac{\gamma}{2}.
$$

\

\noindent On the other hand, a direct computation, using \eqref{asd}, shows that
$$
\rho(x)=\frac{|\nabla G(p,x)|^2}{4G^2(p,x)} \sim C r(x)^{\gamma}\,.
$$
Furthermore, from \eqref{asdf}, the assumption of Theorem \ref{teo2} is satisfied if and only if
$$
\alpha>1-\frac{\gamma}{2},
$$
and the optimality follows in this case.

\

\noindent {\em Case 2:} $\gamma=-2$. We have,
\begin{equation}\label{qwe}
\int_{r}^{\infty}\frac{1}{\varphi(t)^{n-1}}dt = C\, r^{-\delta(n-1)+1}\,.
\end{equation}
Thus
\begin{align*}
 \frac 1{C} \int_{1}^{\infty}  r^{-\delta(n-1)+1}\frac{1}{\big(1+r\big)^{\alpha}}\,r^{\delta(n-1)}\,dr  \leq |u(p)|\leq C  \int_{1}^{\infty} r^{-\delta(n-1)+1}\frac{1}{\big(1+r\big)^{\alpha}}\,r^{\delta(n-1)}\,dr
\end{align*}
Therefore,
\begin{align*}
\frac 1 C\int_{1}^{\infty}\frac{1}{r^{\alpha-1}}\,dr &\leq |u(p)|\leq C \int_{1}^{\infty}\frac{1}{r^{\alpha-1}}\,dr\,,
\end{align*}
and
$$|u(p)|<\infty \quad \textrm{
if and only if} \quad
\alpha>2.
$$

\

\noindent On the other hand, a direct computation, using \eqref{qwe}, shows that
$$
\rho(x)=\frac{|\nabla G(p,x)|^2}{4G^2(p,x)} \sim C r(x)^{-2}\,.
$$
Furthermore, from \eqref{asdf}, the assumption of Theorem \ref{teo2} is satisfied if and only if
$$
\alpha>2,
$$
and the optimality follows in this case.

\

\noindent {\em Case 3:} $\gamma<-2$. We have,
\begin{equation}\label{zxc}
\int_{r}^{\infty}\frac{1}{\varphi(t)^{n-1}}dt = C\, r^{2-n}\,.
\end{equation}
Thus
\begin{align*}
 \frac 1{C} \int_{1}^{\infty}  r^{2-n}\frac{1}{\big(1+r\big)^{\alpha}}\,r^{n-1}\,dr  \leq |u(p)|\leq C  \int_{1}^{\infty} r^{2-n}\frac{1}{\big(1+r\big)^{\alpha}}\,r^{n-1}\,dr
\end{align*}
Therefore,
\begin{align*}
\frac 1 C\int_{1}^{\infty}\frac{1}{r^{\alpha-1}}\,dr &\leq |u(p)|\leq C \int_{1}^{\infty}\frac{1}{r^{\alpha-1}}\,dr\,,
\end{align*}
and
$$|u(p)|<\infty \quad \textrm{
if and only if} \quad
\alpha>2.
$$

\

\noindent On the other hand, a direct computation, using \eqref{zxc}, shows that
$$
\rho(x)=\frac{|\nabla G(p,x)|^2}{4G^2(p,x)} \sim C r(x)^{-2}\,.
$$
Furthermore, from \eqref{asdf}, the assumption of Theorem \ref{teo2} is satisfied if and only if
$$
\alpha>2,
$$
and the optimality follows in this last case.

\

\

\begin{ackn} The authors are members of the Gruppo Nazionale per l'Analisi Matematica, la Probabilit\`{a} e le loro Applicazioni (GNAMPA) of the Istituto Nazionale di Alta Matematica (INdAM). The first two authors are supported by the PRIN project ``Variational methods, with applications to problems in mathematical physics and geometry''.
\end{ackn}

\

\

\

\

\end{document}